\documentclass[11pt]{amsart}
\usepackage[margin=30mm]{geometry}
\usepackage{amsmath,amssymb}
\usepackage{amsthm}
\usepackage{mathrsfs}

\newtheorem{thm}{Theorem}[section]

\newtheorem{lem}{Lemma}[section]
\newtheorem{cor}{Corollary}[section]
\newtheorem{defi}{Definition}[section]

\newtheorem{rem}{Remark}[section]

\usepackage{enumitem}

\begin{document}

\title{A note on odometers and shadowing}
\author{Noriaki Kawaguchi}
\address{Department of Mathematical and Computing Science, School of Computing, Institute of Science Tokyo, 2-12-1 Ookayama, Meguro-ku, Tokyo 152-8552, Japan}
\email{gknoriaki@gmail.com}

\begin{abstract}
For a continuous self-map of a compact metric space, we provide a sufficient condition for the orbit of a point to converge to a periodic orbit or an odometer. We show that if a continuous self-map of a compact metric space has the shadowing property, then the set of points whose orbits converge to periodic orbits or odometers is dense.
\end{abstract}

\maketitle

\markboth{NORIAKI KAWAGUCHI}{A note on odometers and shadowing}

\section{Introduction}

{\em Shadowing} is an important concept in the topological theory of dynamical systems. It was initially introduced in the context of hyperbolic differentiable dynamics \cite{An,B} and generally refers to a situation in which coarse orbits, or {\em pseudo-orbits}, are approximated by true orbits (see \cite{AH} or \cite{P} for background). The topological implications of shadowing have been studied in, e.g., \cite{LO1,LO2,M,MO}. In particular, it is shown in \cite{LO2} that if a continuous self-map of a compact metric space has the shadowing property, then periodic orbits and odometers are dense within its chain recurrent set.

{\em Odometers} are a class of minimal sets in dynamical systems. Together with periodic orbits, they represent regular periodic motion in dynamical systems. In \cite{BK}, an odometer is characterized as an infinite minimal set such that all points are regularly recurrent. In this paper, we study a point whose orbit converges to a periodic point or an odometer. Such an orbit asymptotically resembles  a periodic orbit or an odometer and so exhibits simple asymptotic behavior.

We begin by defining odometers. Given a nondecreasing sequence $m=(m_j)_{j\ge1}$ of positive integers such that $m_j|m_{j+1}$ for each $j\ge1$ and $\lim_{j\to\infty}m_j=\infty$, we define
\begin{itemize}
\item $X(j)=\{0,1,\dots,m_j-1\}$ (with the discrete topology),
\item
\[
X_m=\{(x_j)_{j\ge1}\in\prod_{j\ge1}X(j)\colon x_j\equiv x_{j+1}\pmod{m_j}\:\:\text{for all $j\ge1$}\},
\]
\item $g_m(x)_j=x_j+1\pmod{m_j}$ for all $x=(x_j)_{j\ge1}\in X_m$ and $j\ge1$.
\end{itemize}
We regard $X_m$ as a subspace of the product space $\prod_{j\ge1}X(j)$. The homeomorphism
\[
g_m\colon X_m\to X_m
\]
is called an odometer with the periodic structure $m$. For a continuous self-map $f\colon X\to X$ of a compact metric space $X$, a subset $S$ of $X$ is said to be {\em $f$-invariant} if $f(S)\subset S$. We say that a closed $f$-invariant subset $S$ of $X$ is an {\em odometer} if
\[
f|_S\colon S\to S
\]
is topologically conjugate to an odometer, i.e., there are an odometer $g_m\colon X_m\to X_m$ and a homeomorphism $h\colon X_m\to S$ such that $f|_S\circ h=h\circ g_m$.

Throughout, $X$ denotes a compact metric space endowed with a metric $d$. For a continuous map $f\colon X\to X$ and $x\in X$, the {\em $\omega$-limit set} $\omega(x,f)$ of $x$ for $f$ is defined as the set of $y\in X$ such that
\[
\lim_{j\to\infty}f^{i_j}(x)=y
\]
for some sequence $0\le i_1<i_2<\cdots$. Note that
\[
\lim_{i\to\infty}d(f^i(x),\omega(x,f))=0
\]
for all $x\in X$. We say that a closed $f$-invariant subset $S$ of $X$ is a {\em periodic orbit} if there are $x\in S$ and $i>0$ such that $f^i(x)=x$ and $S=\{x,f(x),\dots,f^{i-1}(x)\}$. We denote by $\mathbb{A}(f)$ the set of $x\in X$ such that $\omega(x,f)$ is a periodic orbit or an odometer:
\[
\mathbb{A}(f)=\{x\in X\colon\text{$\omega(x,f)$ is a periodic orbit or an odometer}\}.
\]

\begin{rem}
\normalfont
For any continuous map $f\colon X\to X$ and $x\in X$, if $x\in\mathbb{A}(f)$, then
\[
\lim_{i\to\infty}d(f^i(x),f^i(y))=0
\]
for some $y\in\omega(x,f)$.
\end{rem}

The first result gives a sufficient condition for the orbit of a point to converge to a periodic orbit or an odometer.

\begin{thm}
Let $f\colon X\to X$ be a continuous map. Given any $x\in X$, if there are $x_{\epsilon}\in\mathbb{A}(f)$, $\epsilon>0$, such that
\[
\limsup_{i\to\infty}d(f^i(x),f^i(x_\epsilon))\le\epsilon
\]
for all $\epsilon>0$, then $x\in\mathbb{A}(f)$.
\end{thm}

For a continuous map $f\colon X\to X$ and $x\in X$, we say that $f$ is {\em equicontinuous} at $x$ if for every $\epsilon>0$, there is $\delta>0$ such that $d(x,y)\le\delta$ implies
\[
\sup_{i\ge0}d(f^i(x),f^i(y))\le\epsilon
\]
for all $y\in X$. From Theorem 1.1, we obtain the following corollary. 

\begin{cor}
Let $f\colon X\to X$ be a continuous map and let $x\in X$. If $f$ is equicontinuous at $x$ and if $x\in\overline{\mathbb{A}(f)}$, then $x\in\mathbb{A}(f)$.
\end{cor}

Let us recall the notion of chain recurrence.
 
\begin{defi}
\normalfont
Let $f\colon X\to X$ be a continuous map. For $\delta>0$, a finite sequence $(x_i)_{i=0}^k$, $k>0$, of points in $X$ is called a {\em $\delta$-chain} of $f$ if $d(f(x_i),x_{i+1})\le\delta$ for all $0\le i\le k-1$. We say that $x\in X$ is a {\em chain recurrent point} for $f$ if for every $\delta>0$, there is a $\delta$-chain $(x_i)_{i=0}^k$ of $f$ such that $x_0=x_k=x$. We denote by $CR(f)$ the set of chain recurrent points for $f$.
\end{defi}

Next we recall the basic definition of minimality.

\begin{defi}
\normalfont
Let $f\colon X\to X$ be a continuous map. A closed $f$-invariant subset $M$ of $X$ is said to be a {\em minimal set} for $f$ if closed $f$-invariant subsets of $M$ are only $\emptyset$ and $M$. This is equivalent to $M=\overline{\{f^i(x)\colon i\ge0\}}$ for all $x\in M$. The map $f\colon X\to X$ is said to be {\em minimal} if $X$ is a minimal set for $f$. We say that $x\in X$ is a {\em minimal point} for $f$ if $\overline{\{f^i(x)\colon i\ge0\}}$ is a minimal set for $f$. We denote by $M(f)$ the set of minimal points for $f$.
\end{defi}

\begin{rem}
\normalfont
\begin{itemize}
\item Zorn's lemma implies $M(f)\ne\emptyset$.
\item Let $f\colon X\to X$ be a continuous map and let $x\in X$. We know that $x\in M(f)$ if and only if \[
\{i\ge0\colon d(x,f^i(x))\le\epsilon\}
\]
is {\em syndetic} for all $\epsilon>0$, i.e., for any $\epsilon>0$, there is $n\ge1$ such that
\[
\{i\ge0\colon d(x,f^i(x))\le\epsilon\}\cap\{k,k+1,\dots,k+n-1\}\ne\emptyset
\] 
for all $k\ge0$.
\end{itemize}
\end{rem}

Then we recall the definition of regular recurrence.

\begin{defi}
\normalfont
Let $f\colon X\to X$ be a continuous map. We say that $x\in X$ is a {\em regularly recurrent point} for $f$ if for every $\epsilon>0$, there exists $n>0$ such that $d(x,f^{kn}(x))\le\epsilon$ for all $k\ge0$. We denote by $RR(f)$ the set of regularly recurrent points for $f$.
\end{defi}

\begin{rem}
\normalfont
\begin{itemize}
\item Let $f\colon X\to X$ be a continuous map. We say that $x\in X$ is a {\em periodic point} for $f$ if $f^i(x)=x$ for some $i>0$. We denote by $Per(f)$ the set of periodic points for $f$. Note that
\[
Per(f)\subset RR(f)\subset M(f).
\]
\item An odometer $g_m\colon X_m\to X_m$ is minimal and satisfies $X_m=RR(g_m)$. By Corollary 2.5 of \cite{BK}, we know that a minimal continuous map $f\colon X\to X$ satisfies $X=RR(f)$ if and only if $X$ is a periodic orbit or an odometer.
\end{itemize}
\end{rem}

In order to state the other main result, we recall the definition of shadowing.

\begin{defi}
\normalfont
Let $f\colon X\to X$ be a continuous map and let $\xi=(x_i)_{i\ge0}$ be a sequence of points in $X$. For $\delta>0$, $\xi$ is called a {\em $\delta$-pseudo orbit} of $f$ if $d(f(x_i),x_{i+1})\le\delta$ for all $i\ge0$. For $\epsilon>0$, $\xi$ is said to be {\em $\epsilon$-shadowed} by $x\in X$ if $d(f^i(x),x_i)\leq \epsilon$ for all $i\ge 0$. We say that $f$ has the {\em shadowing property} if for any $\epsilon>0$, there is $\delta>0$ such that every $\delta$-pseudo orbit of $f$ is $\epsilon$-shadowed by some point of $X$. 
\end{defi}

In \cite{M}, it is shown that every continuous map $f\colon X\to X$ with the shadowing property satisfies $CR(f)=\overline{M(f)}$; and a question is raised whether every such map $f\colon X\to X$ satisfies $CR(f)=\overline{RR(f)}$. Theorem 3.2 in \cite{MO} gives a positive answer to this question. Moreover, in \cite{LO2}, it is shown that if a continuous map $f\colon X\to X$ has the shadowing property, then
\[
CR(f)=\overline{\{x\in\mathbb{A}(f)\colon x\in\omega(x,f)\}}
\]
(see Corollary 3.3 of \cite{LO2}).

The second main result is the following.

\begin{thm}
If a continuous map $f\colon X\to X$ has the shadowing property, then $X=\overline{\mathbb{A}(f)}$, i.e., $\mathbb{A}(f)$ is a dense subset of $X$.
\end{thm}

This paper consists of four sections. Our proofs of Theorems 1.1 and 1.2 exploit the characterization of an odometer as an infinite minimal set in which all points are regularly recurrent (\cite[Corollary 2.5]{BK}). For the sake of completeness, in Section 2, we give a simple alternative proof of this fact. Theorems 1.1 and 1.2 are proved in Section 3. In Section 4, we recall the notion of chain continuity and make some remarks on its relevance to the main results.

\section{Preliminaries}

Our proofs of Theorems 1.1 and 1.2 will be based on the characterization of an odometer as an infinite minimal set such that all points are regularly recurrent (\cite[Corollary 2.5]{BK}). In this section, we give a simple alternative proof of this fact.

A subspace $S$ of $X$ is said to be {\em totally disconnected} if every connected component of $S$ is a singleton. A metric $d$ on $X$ is called an {\em ultrametric} on $X$ if
\[
d(x,z)\le\max\{d(x,y),d(y,z)\}
\]
for all $x,y,z\in X$. We know that a compact metric space $X$ endowed with a metric $d$ is totally disconnected if and only if there is an ultrametic $D$ on $X$ equivalent to $d$.

A map $f\colon X\to X$ is said to be {\em equicontinuous} if for any $\epsilon>0$, there is $\delta>0$ such that $d(x,y)\le\delta$ implies
\[
\sup_{i\ge0}d(f^i(x),f^i(y))\le\epsilon
\]
for all $x,y\in X$. It is known that if an equicontinuous map $f\colon X\to X$ is surjective, then $f$ is a homeomorphism and $f^{-1}$ is also equicontinuous (cf. \cite{AG, Ma}).

For $x\in X$ and $\epsilon>0$, let $B_\epsilon(x)$ denote the closed $\epsilon$-ball centered at $x$:
\[
B_\epsilon(x)=\{y\in X\colon d(x,y)\le\epsilon\}.
\]
For a subset $S$ of $X$ and $\epsilon>0$, we denote by $B_\epsilon(S)$ the $\epsilon$-neighborhood of $S$:
\[
B_\epsilon(S)=\{x\in X\colon d(x,S)=\inf_{y\in S}d(x,y)\le\epsilon\}.
\]
If $d$ is an ultrametric, then $B_\epsilon(x)=B_\epsilon(y)$ for all $\epsilon>0$ and $x,y\in X$ with $d(x,y)\le\epsilon$. This implies that for any $\epsilon>0$ and $x,y\in X$, if $B_\epsilon(x)\cap B_\epsilon(y)\ne\emptyset$, then $B_\epsilon(x)=B_\epsilon(y)$.

A proof of the following lemma can be found also in \cite{Ku} (see Theorem 4.4 of \cite{Ku}).

\begin{lem}
If $X$ is totally disconnected, then for any minimal equicontinuous homeomorphism $f\colon X\to X$, $X$ is a periodic orbit or an odometer.
\end{lem}

\begin{proof}
Since $X$ is totally disconnected, we may assume that $d$ is an ultrametric on $X$. Since $f\colon X\to X$ is an equicontinuous homeomorphism, if necessary, replacing $d$ by an equivalent metric $D$ defined as
\[
D(x,y)=\sup_{i\in\mathbb{Z}}d(f^i(x),f^i(y))
\]
for all $x,y\in X$, we may also assume that $d(f(x),f(y))=d(x,y)$ for all $x,y\in X$. Fix $x\in X$ and note that
\[
f^i(B_\epsilon(x))=B_\epsilon(f^i(x))
\]
for all $\epsilon>0$ and $i\ge0$.  Given any $\epsilon>0$, since $f$ is minimal, we have
\[
B_\epsilon(x)\cap f^i(B_\epsilon(x))\ne\emptyset
\]
for some $i>0$. Letting
\[
l_\epsilon=\min\{i>0\colon B_\epsilon(x)\cap f^i(B_\epsilon(x))\ne\emptyset\},
\]
we obtain
\begin{itemize}
\item[(1)] $B_\epsilon(x)\cap f^{l_\epsilon}(B_\epsilon(x))\ne\emptyset$,
\item[(2)] $f^i(B_\epsilon(x))\cap f^j(B_\epsilon(x))=\emptyset$ for all $0\le i<j<l_\epsilon$.
\end{itemize}
From (1), it follows that
\[
B_\epsilon(x)\cap B_\epsilon(f^{l_\epsilon}(x))=B_\epsilon(x)\cap f^{l_\epsilon}(B_\epsilon(x))\ne\emptyset
\]
and so $B_\epsilon(x)=B_\epsilon(f^{l_\epsilon}(x))=f^{l_\epsilon}(B_\epsilon(x))$. Since $f$ is minimal and
\[
\bigsqcup_{i=0}^{l_\epsilon-1}f^i(B_\epsilon(x))
\]
is a closed $f$-invariant subset of $X$, we obtain
\[
X=\bigsqcup_{i=0}^{l_\epsilon-1}f^i(B_\epsilon(x)).
\]
We fix a sequence $\epsilon_j>0$, $j\ge1$, such that $\epsilon_1>\epsilon_2>\cdots$ and $\lim_{j\to\infty}\epsilon_j=0$. Note that
\[
f^i(B_{\epsilon_{j+1}}(x))\subset f^i(B_{\epsilon_j}(x))
\]
for all $j\ge1$ and $i\ge0$. Letting $m_j=l_{\epsilon_j}$, $j\ge1$, we have $m_j|m_{j+1}$ for all $j\ge1$. If
\[
\sup_{j\ge1}m_j<\infty,
\]
then $X$ is clearly a periodic orbit. Consider the case where
\[
\lim_{j\to\infty}m_j=\infty.
\]
Following the notation in Section 1, we define a map $h\colon X_m\to X$ by: for all $y\in X$ and $z=(z_j)_{j\ge1}\in X_m$, $h(z)=y$ if and only if 
\[
y\in\bigcap_{j\ge1}f^{z_j}(B_{\epsilon_j}(x))=\bigcap_{j\ge1}B_{\epsilon_j}(f^{z_j}(x)).
\]
Since $h$ is a homeomorphism and satisfies
\[
f\circ h=h\circ g_m,
\]
we conclude that $X$ is an odometer, completing the proof.
\end{proof}

A {\em partition} of $X$ is a family of disjoint subsets of $X$ whose union is $X$. For a subset $S$ of $X$ and a partition $\mathcal{P}$ of $X$, the notation $S\prec\mathcal{P}$ means $S\subset A$ for some $A\in\mathcal{P}$. By using Lemma 2.1, we prove the following (see Corollary 2.5 of \cite{BK}).

\begin{lem}
If a minimal continuous map $f\colon X\to X$ satisfies $X=RR(f)$, then $X$ is a periodic orbit or an odometer.
\end{lem}

\begin{proof}
Let $f\colon X\to X$ be a minimal continuous map with $X=RR(f)$. Due to Lemma 2.1, it suffices to show that $X$ is totally disconnected and $f$ is equicontinuous. Fix $\epsilon>0$. Given any $x\in X$, since $x\in RR(f)$, we have
\[
\overline{\{f^{kn_x}(x)\colon k\ge0\}}\subset B_\epsilon(x) 
\]
for some $n_x>0$. Let
\[
A_x=\overline{\{f^{kn_x}(x)\colon k\ge0\}}
\]
and note that $A_x$ is a closed $f^{n_x}$-invariant subset of $X$. Since
\[
x\in RR(f)=RR(f^{n_x})\subset M(f^{n_x}),
\]
$A_x$ is a minimal set for $f^{n_x}$. Since $f^i(A_x)$ is a minimal set for $f^{n_x}$ for all $i\ge0$, $f^i(A_x)\cap f^j(A_x)\ne\emptyset$ implies $f^i(A_x)=f^j(A_x)$ for all $0\le i<j$. Letting
\[
l_x=\min\{l>0\colon A_x\cap f^l(A_x)\ne\emptyset\},
\]
we obtain
\begin{itemize}
\item[(1)] $f^{l_x}(A_x)=A_x$,
\item[(2)] $f^i(A_x)\cap f^j(A_x)=\emptyset$ for all $0\le i<j<l_x$.
\end{itemize}
Since $f$ is minimal and
\[
\bigsqcup_{i=0}^{l_x-1}f^i(A_x)
\]
is a closed $f$-invariant subset of $X$, we have
\[
X=\bigsqcup_{i=0}^{l_x-1}f^i(A_x).
\]
Let
\[
\mathcal{P}_x=\{f^i(A_x)\colon 0\le i\le l_x-1\}
\]
and note that $\mathcal{P}_x$ is a clopen partition of $X$. Note also that $x\in A_x$ and ${\rm diam}A_x\le2\epsilon$. Since $x\in X$ and $\epsilon>0$ are arbitrary, $X$ is totally disconnected. The form of $\mathcal{P}_x$ implies that  for any subset $S$ of $X$, if $S\prec\mathcal{P}_x$, then $f(S)\prec\mathcal{P}_x$. We take $x_1,x_2,\dots,x_m\in X$ such that
\[
X=\bigcup_{j=1}^m A_{x_j}.
\]
If $\delta>0$ is sufficiently small, then given any $x,y\in X$ with $d(x,y)\le\delta$, we have
\[
\{x,y\}\prec\mathcal{P}_{x_j}
\]
for all $1\le j\le m$. It follows that
\[
\{f^i(x),f^i(y)\}\prec\mathcal{P}_{x_j}
\]
for all $1\le j\le m$ and $i\ge0$. For each $i\ge0$, we have $f^i(x)\in A_{x_{j_i}}$ for some $1\le j_i\le m$. This with
\[
\{f^i(x),f^i(y)\}\prec\mathcal{P}_{x_{j_i}}
\]
implies $f^i(y)\in A_{x_{j_i}}$ and so $d(f^i(x),f^i(y))\le{\rm diam}A_{x_{j_i}}\le2\epsilon$ for all $i\ge0$. Since $\epsilon>0$ is arbitrary, we  conclude that $f$ is equicontinuous, completing the proof. 
\end{proof}

\section{Proofs of Theorems 1.1 and 1.2}

In this section, we prove the main results of this paper (Theorems 1.1 and 1.2). We first prove a simple lemma. For a sequence $\xi=(x_i)_{i\ge0}$ of points in $X$, we define the $\omega$-limit set $\omega(\xi)$ as the set of $y\in X$ such that
\[
\lim_{j\to\infty}x_{i_j}=y
\]
for some sequence $0\le i_1<i_2<\cdots$. Note that for any continuous map $f\colon X\to X$ and $x\in X$, we have $\omega(x,f)=\omega((f^i(x))_{i\ge0})$. 

\begin{lem}
Let $f\colon X\to X$ be a continuous map. For any $\epsilon>0$ and two sequences $\xi=(x_i)_{i\ge0}$, $\xi'=(y_i)_{i\ge0}$ of points in $X$, if
\[
\limsup_{i\to\infty}d(x_i,y_i)\le\epsilon,
\]
then $\omega(\xi)\subset B_\epsilon(\omega(\xi'))$ and $\omega(\xi')\subset B_\epsilon(\omega(\xi))$.
\end{lem}

\begin{proof}
For any $z\in\omega(\xi)$, there are a sequence $0\le i_1<i_2<\cdots$ and $w\in X$ such that $\lim_{j\to\infty}x_{i_j}=z$ and $\lim_{j\to\infty}y_{i_j}=w$. Note that $w\in\omega(\xi')$. By
\[
\limsup_{i\to\infty}d(x_i,y_i)\le\epsilon,
\]
we obtain $d(z,w)\le\epsilon$ and so $z\in B_\epsilon(\omega(\xi'))$. Since $z\in\omega(\xi)$ is arbitrary, we obtain $\omega(\xi)\subset B_\epsilon(\omega(\xi'))$. The proof of $\omega(\xi')\subset B_\epsilon(\omega(\xi))$ is similar.
\end{proof}

Let us prove Theorem 1.1.

\begin{proof}[Proof of Theorem 1.1]
We fix $x_{\epsilon}\in\mathbb{A}(f)$, $\epsilon>0$, such that
\[
\limsup_{i\to\infty}d(f^i(x),f^i(x_\epsilon))\le\epsilon
\]
for all $\epsilon>0$. Due to Lemma 2.2, it suffices to show that $\omega(x,f)$ is a minimal set for $f$ and satisfies $\omega(x,f)\subset RR(f)$. For any $z\in\omega(x,f)$ and $\epsilon>0$, there are a sequence $0\le i_1<i_2<\cdots$ and $w\in X$ such that $\lim_{j\to\infty}f^{i_j}(x)=z$ and $\lim_{j\to\infty}f^{i_j}(x_\epsilon)=w$. Note that $w\in\omega(x_\epsilon,f)$.  By
\[
\limsup_{i\to\infty}d(f^i(x),f^i(x_\epsilon))\le\epsilon,
\]
we obtain
\[
\sup_{i\ge0}d(f^i(z),f^i(w))\le\epsilon.
\]
Since $x_\epsilon\in\mathbb{A}(f)$, $\omega(x_\epsilon,f)$ is a minimal set for $f$. By $w\in\omega(x_\epsilon,f)$, we obtain
\[
\omega(x_\epsilon,f)=\omega(w,f).
\]
From Lemma 3.1, it follows that
\[
\omega(x,f)\subset B_\epsilon(\omega(x_\epsilon,f))=B_\epsilon(\omega(w,f))\subset B_{2\epsilon}(\omega(z,f)).
\]
On the other hand, $x_\epsilon\in\mathbb{A}(f)$ implies
\[
w\in\omega(x_\epsilon,f)\subset RR(f). 
\]
We take $n>0$ such that
\[
\sup_{k\ge0}d(w,f^{kn}(w))\le\epsilon.
\] 
It follows that
\[
d(z,f^{kn}(z))\le d(z,w)+d(w,f^{kn}(w))+d(f^{kn}(w),f^{kn}(z))\le\epsilon+\epsilon+\epsilon=3\epsilon
\]
for all $k\ge0$, thus
\[
\sup_{k\ge0}d(z,f^{kn}(z))\le3\epsilon.
\]
Since $z\in\omega(x,f)$ and $\epsilon>0$ are arbitrary, we obtain $\omega(x,f)=\omega(z,f)$ and $z\in RR(f)$ for all $z\in\omega(x,f)$. We conclude that $\omega(x,f)$ is a minimal set for $f$ and satisfies $\omega(x,f)\subset RR(f)$, proving the theorem.
\end{proof}

For the proof of Theorem 1.2, we need a lemma that provides, in terms of pseudo-orbits, a sufficient condition for $x\in X$ to be approximated by $y\in\mathbb{A}(f)$ where $f\colon X\to X$ is a continuous map.  For a continuous map $f\colon X\to X$ and $\delta>0$, we say that a $\delta$-pseudo orbit $\xi=(x_i)_{i\ge0}$ of $f$ is {\em eventually periodic} if there are $k\ge0$ and $n>0$ such that $x_i=x_{i+n}$ for all $i\ge k$.

\begin{lem}
Let $f\colon X\to X$ be a continuous map. Let $\epsilon>0$ and let $\epsilon_j,\delta_j>0$, $j\ge0$, be two sequences of positive numbers such that
\begin{itemize}
\item[(1)] $\sum_{j=0}^\infty\epsilon_j\le\epsilon$,
\item[(2)] $\lim_{j\to\infty}\delta_j=0$.
\end{itemize}
Let $\xi_j=(x_i^{(j)})_{i\ge0}$, $j\ge0$, be a sequence of $\delta_j$-pseudo orbits of $f$ such that
\begin{itemize}
\item[(3)] for every $j\ge0$, $\xi_j$ is eventually periodic,
\item[(4)] $\sup_{i\ge0}d(x_i^{(j)},x_i^{(j+1)})\le\epsilon_j$ for all $j\ge0$.
\end{itemize}
Let $x_0=\lim_{j\to\infty}x_0^{(j)}$. Then, $x_0\in\mathbb{A}(f)$ and $d(x_0^{(0)},x_0)\le\epsilon$.
\end{lem}

\begin{proof}
By $(1)$ and $(4)$, $(x_i^{(j)})_{j\ge0}$ is a Cauchy sequence for all $i\ge0$. Let $x_i=\lim_{j\to\infty}x_i^{(j)}$ for each $i\ge0$. Let $\zeta_j=\sum_{a=j}^\infty\epsilon_a$, $j\ge0$, and note that
\[
\sup_{i\ge0}d(x_i^{(j)},x_i)\le\zeta_j
\]
for all $j\ge0$, in particular, $d(x_0^{(0)},x_0)\le\zeta_0\le\epsilon$. For any $i\ge0$, since
\begin{align*}
d(f(x_i),x_{i+1})&\le d(f(x_i),f(x_i^{(j)}))+d(f(x_i^{(j)}),x_{i+1}^{(j)})+d(x_{i+1}^{(j)},x_{i+1})\\
&\le d(f(x_i),f(x_i^{(j)}))+\delta_j+d(x_{i+1}^{(j)},x_{i+1})
\end{align*}
for all $j\ge0$, by $(2)$, we obtain $f(x_i)=x_{i+1}$. By (3), there are $k_j\ge0$, $n_j>0$, $j\ge0$, such that
\[
x_i^{(j)}=x_{i+n_j}^{(j)}
\]
for all $j\ge0$ and $i\ge k_j$. Given any $z\in\omega(x_0,f)$ and $j\ge0$, there are a sequence $k_{j}\le l_1<l_2<\cdots$ and $k_j\le i_j<k_j+n_j$ such that
\begin{itemize}
\item $\lim_{m\to\infty}f^{l_m}(x_0)=z$,
\item $l_m\equiv i_j\pmod{n_j}$ for all $m\ge1$.
\end{itemize}
It follows that
\[
x_{l_m+\beta}^{(j)}=x_{i_j+\beta}^{(j)}
\]
for all $m\ge1$ and $\beta\ge0$. We obtain
\[
\sup_{\beta\ge0}d(x_{i_j+\beta}^{(j)},f^\beta(f^{l_m}(x_0)))=\sup_{\beta\ge0}d(x_{l_m+\beta}^{(j)},x_{l_m+\beta})\le\zeta_j
\]
for all $m\ge1$. Letting $m\to\infty$, we obtain
\[
\sup_{\beta\ge0}d(x_{i_j+\beta}^{(j)},f^\beta(z))\le\zeta_j.
\]
Let $\xi=(f^i(x_0))_{i\ge0}=(x_i)_{i\ge0}$ and $\xi'_j=(x_{i_j+\beta}^{(j)})_{\beta\ge0}$. By Lemma 3.1, we obtain
\[
\omega(x_0,f)=\omega(\xi)\subset B_{\zeta_j}(\omega(\xi_j))=B_{\zeta_j}(\omega(\xi'_j))\subset B_{2\zeta_j}(\omega(z,f)).
\]
On the other hand, for every $k\ge0$, we have
\begin{align*}
d(z,f^{kn_j}(z))&\le d(z,x_{i_j}^{(j)})+d(x_{i_j}^{(j)},x_{i_j+kn_j}^{(j)})+d(x_{i_j+kn_j}^{(j)},f^{kn_j}(z))\\
&\le\zeta_j+0+\zeta_j=2\zeta_j.
\end{align*}
Note that $\lim_{j\to\infty}\zeta_j=0$. Since $z\in\omega(x_0,f)$ and $j\ge0$ are arbitrary, we obtain $\omega(x_0,f)=\omega(z,f)$ and $z\in RR(f)$ for all $z\in\omega(x_0,f)$. We conclude that $\omega(x_0,f)$ is a minimal set for $f$ and satisfies $\omega(x_0,f)\subset RR(f)$. By Lemma 2.2, we obtain $x_0\in\mathbb{A}(f)$, proving the lemma.
\end{proof}

In fact, Theorem 1.2 is a direct consequence of Lemma 3.2 and the following lemma.

\begin{lem}
Let $f\colon X\to X$ be a continuous map. Let $\epsilon_j,\delta_j>0$, $j\ge0$, be two sequences of positive numbers such that
\begin{itemize}
\item[(1)] for any $j\ge0$, every $\delta_j$-pseudo orbit of $f$ is $\epsilon_j$-shadowed by some point of $X$.
\end{itemize}
Let $\xi_0=(x_i^{(0)})_{i\ge0}$ be an eventually periodic $\delta_0$-pseudo orbit of $f$. Then, there is a sequence $\xi_j=(x_i^{(j)})_{i\ge0}$, $j\ge0$, of $\delta_j$-pseudo orbits of $f$ such that
\begin{itemize}
\item[(2)] for every $j\ge0$, $\xi_j$ is eventually periodic,
\item[(3)] $\sup_{i\ge0}d(x_i^{(j)},x_i^{(j+1)})\le\epsilon_j$ for all $j\ge0$.
\end{itemize}
\end{lem}

\begin{proof}
The proof is by induction on $j$. Assume that $\xi_j$ is given for some $j\ge0$. Since $\xi_j$ is eventually periodic, there are $k\ge0$ and $n>0$ such that
\[
x_i^{(j)}=x_{i+n}^{(j)}
\]
for all $i\ge k$. Since $\xi_j$ is a $\delta_j$-pseudo orbit of $f$, by (1), we obtain $y\in X$ such that
\[
\sup_{i\ge0}d(f^i(y),x_i^{(j)})\le\epsilon_j.
\]
By taking $z\in\omega(f^k(y),f^n)$, we obtain a sequence $0\le l_1<l_2<\cdots$ such that
\[
\lim_{a\to\infty}f^{k+l_a n}(y)=\lim_{a\to\infty}f^{l_a n}(f^k(y))=z.
\]
It follows that
\[
d(f^{k+ln}(y),f^{k+mn}(y))\le\delta_{j+1}
\]
for some $0\le l<m$. Let
\[
\alpha=(y,f(y),\dots,f^{k+ln}(y))
\]
and let
\[
\beta=(f^{k+ln}(y),f^{k+ln+1}(y),\dots,f^{k+mn-1}(y),f^{k+ln}(y)).
\]
Letting
\[
\xi_{j+1}=(x_i^{(j+1)})_{i\ge0}=\alpha\beta\beta\beta\cdots,
\]
we see that
\begin{itemize}
\item $\xi_{j+1}$ is a $\delta_{j+1}$-pseudo orbit of $f$,
\item $\xi_{j+1}$ is eventually periodic,
\item $\sup_{i\ge0}d(x_i^{(j)},x_i^{(j+1)})\le\epsilon_j$.
\end{itemize}
Thus, the induction is complete.
\end{proof}

Finally, we prove Theorem 1.2.

\begin{proof}[Proof of Theorem 1.2]
Given any $x\in X$ and $\epsilon>0$, we fix a sequence $\epsilon_j>0$, $j\ge0$, of positive numbers  such that
\begin{itemize}
\item $\sum_{j=0}^\infty\epsilon_j\le\epsilon$.
\end{itemize}
Since $f$ has the shadowing property, we can take a sequence $\delta_j>0$, $j\ge0$, of positive numbers such that
\begin{itemize}
\item $\lim_{j\to\infty}\delta_j=0$,
\item for any $j\ge0$, every $\delta_j$-pseudo orbit of $f$ is $\epsilon_j$-shadowed by some point of $X$.
\end{itemize}
By Lemma 3.3, taking an eventually periodic $\delta_0$-pseudo orbit $\xi_0=(x_i^{(0)})_{i\ge0}$ of $f$ with $x_0^{(0)}=x$, we obtain a sequence $\xi_j=(x_i^{(j)})_{i\ge0}$, $j\ge0$, of $\delta_j$-pseudo orbits of $f$ such that
\begin{itemize}
\item for every $j\ge0$, $\xi_j$ is eventually periodic,
\item $\sup_{i\ge0}d(x_i^{(j)},x_i^{(j+1)})\le\epsilon_j$ for all $j\ge0$.
\end{itemize}
By Lemma 3.2, letting $x_0=\lim_{j\to\infty}x_0^{(j)}$, we obtain $x_0\in\mathbb{A}(f)$ and 
\[
d(x,x_0)=d(x_0^{(0)},x_0)\le\epsilon.
\]
Since $x\in X$ and $\epsilon>0$ are arbitrary, we conclude that $X=\overline{\mathbb{A}(f)}$, completing the proof of the theorem.
\end{proof}

\section{Remarks}

In this final section, we recall the notion of chain continuity and make some remarks on its relevance to the main results.

\begin{defi}
\normalfont
Given a continuous map $f\colon X\to X$ and $x\in X$, we say that $f$ is {\em chain continuous} at $x$ if for any $\epsilon>0$, there is $\delta>0$ such that every $\delta$-pseudo orbit $(x_i)_{i\ge0}$ of $f$ with $x_0=x$ is $\epsilon$-shadowed by $x$ \cite{A}. We denote by $CC(f)$ the set of chain continuity points for $f$. 
\end{defi}

\begin{defi}
\normalfont
Let $f\colon X\to X$ be a continuous map. We say that a closed $f$-invariant subset $S$ of $X$ is {\em chain stable} if for any $\epsilon>0$, there is $\delta>0$ such that every $\delta$-chain $(x_i)_{i=0}^k$ of $f$ with $x_0\in S$ satisfies $d(x_k,S)=\inf_{y\in S}d(x_k,y)\le\epsilon$.
\end{defi}

Let $f\colon X\to X$ be a continuous map. A {\em chain component} for $f$ is by definition a maximal internally  chain transitive subset of $X$. By Theorem 7.5 of \cite{AHK}, we know that for any $x\in X$, $x\in CC(f)$ if and only if $\omega(x,f)$ coincides with a {\em terminal} (i.e., chain stable) chain component for $f$ and is a periodic orbit or an odometer (see also Introduction of \cite{K}). In particular, we have
\[
CC(f)\subset\mathbb{A}(f).
\]
For any $x,y\in X$, if $x\in CC(f)$ and
\[
\liminf_{i\to\infty}d(f^i(x),f^i(y))=0,
\] 
then we easily see that $y\in CC(f)$ and
\[
\limsup_{i\to\infty}d(f^i(x),f^i(y))=0.
\]

We say that a continuous map $f\colon X\to X$ is {\em almost chain continuous} if $CC(f)$ is a dense $G_\delta$-subset of $X$.  By Theorem 1.2 of \cite{K}, we know that for any continuous map $f\colon X\to X$, if $X$ is locally connected, $f$ has the shadowing property, and if $CR(f)$ is totally disconnected, then $f$ is almost chain continuous. As a corollary, generic homeomorphisms or continuous self-maps of a closed differentiable manifold are almost chain continuous (see Corollary 1.1 of \cite{K}). Note that if a continuous map $f\colon X\to X$ is almost chain continuous, then since
\[
CC(f)\subset\mathbb{A}(f),
\]
$\mathbb{A}(f)$ is a residual subset of $X$.

\end{document}